\documentclass[a4paper,10pt]{article}

\usepackage[english]{babel}
\usepackage{amsmath,amssymb,amsthm}
\usepackage{mathrsfs}

 \usepackage[noadjust]{cite} 
\usepackage{cancel}
\usepackage{dashbox}
\usepackage{xcolor}
\newcommand{\E}[1]{\mathrm{E} \Big[ #1 \Big]}

\RequirePackage{geometry}
\def\C{\mathcal{C}}

\newtheorem{theorem}{Theorem}
\newtheorem{proposition}[theorem]{Proposition}

\theoremstyle{definition}
\newtheorem{definition}[theorem]{Definition}
\newtheorem{remark}[theorem]{Remark}

\author{%
	Francesco C. De Vecchi\footnote{University of Bonn, Germany
		- {	fdevecchi@uni-bonn.de }}
	\and 
	Luca Giordano\footnote{University of Milano, Italy
		- {luca.giordano,daniela.morale,stefania.ugolini@unimi.it}}  
	\and Daniela Morale\footnotemark[3] \and Stefania Ugolini\footnotemark[3]
	}

\title{A note on the continuity in the Hurst index of the solution \\of rough differential equations driven \\by a fractional Brownian motion}

\begin{document}
 
\maketitle

\abstract{
Within the rough path framework	 we prove the continuity of the solution to random differential equations  driven by fractional Brownian motion with respect to  the Hurst parameter $H$  when $H\in (1/3,1/2]$. 
}

\section{Introduction}

The importance of the study of stochastic equations driven by a fractional Brownian motion with parameter $H\in(0,1)$ naturally arises from the observation of many phenomena for which the assumption of independence of increments which is intrinsic, for example in the case of the standard Brownian motion, cannot be supposed (\cite{Hurst}). Indeed, in biology, meteorology, telecommunications, queueing
theory  and finance evidence of memory and autocorrelation effects are shown (\cite{giordano_morale,ionescu_2017_fBm_biology,rostek_2013_fBm_finance}). The estimation of $H$ is very important, since it determines the magnitude of the self-correlation of the noise in the models.
 As emphasized in \cite{Jolis}, not only one has to deal with the problem of the estimation of the Hurst parameter $H$ of the noise, as in \cite{kubilius_estimation_H,mielniczuk_estimation_H,nualart}, but one
needs to check that the model does not exhibit a large sensitivity with respect to the values of $H$.   Hence, the study of the continuity problem is important   in the case of  both time (SDE) and time-space (SPDE) stochastic differential equations driven by fractional noises, and it is  a very interesting problem not only from a theoretical point of view, but also in the modeling applications (\cite{giordano:bernoulli_2019,Hairer book,Friz book,RiTa16,RiTa17}). 

\medskip

Here we investigate a continuity problem for a stochastic differential equation (SDE) driven by a fractional Brownian motion $Y^H$ in the rough paths theory setting (\cite{Gubinelli 2004,{Hairer book},{Friz book},{Friz paper}}). In particular, the central object is the following equation
\begin{equation}
  d X_t  = \alpha  (X_t ) d t +   \beta  (X_t)
  \circ d Y_t, \label{eq:general}
\end{equation}
where $Y : [0, T] \rightarrow \mathbb R^d$ is a driving signal, $X : [0; T] \rightarrow \mathbb R^n$, $\alpha : \mathbb R^n \rightarrow \mathbb R^n$, and $\beta:\mathbb R^n \rightarrow \mathbb R^{n \times d}$ are   smooth functions. As usual the solution process $\{X_t\}_{t\in[0,T]}$ has to be interpreted  in the integral form, given an initial datum $X_0$ , i.e.
\begin{equation}
  X_t  = X_0+ \int_0^t \alpha  (X_t ) d t +  \int_0^t  \beta  (X_t)
\circ d Y_t, \label{eq:general_integral_form}
\end{equation}
where the integral has to be understood in the sense of canonical rough integral. We recall that integration in the rough path sense involves functions with low regularity, in particular which are H\"older continuous.

\smallskip
We consider the particular case in which the noise is given by a fractional Brownian motion, that is the process $Y$ in \eqref{eq:general} is  a fractional Brownian motion  $W^H$ with Hurst parameter $H\in(0,1)$.
The continuity of solution of stochastic differential equations driven by a fractional Brownian motion  and its functional, linear and not linear, has been already investigated in \cite{RiTa16,RiTa17}, in the case   $H\downarrow \frac12$. We recall that when $H>1/2$   the  integral in equation \eqref{eq:general_integral_form} is a pathwise Stieltjes integral in the sense of Young (\cite{young}). In the applications, however, the estimation  of $ H$  shows that it may take values less then $1/2$ (\cite{giordano_morale}). In the present work we consider the case $H\rightarrow  H_\infty\in \left(\frac13, \frac12 \right]$  and we specialize to the case of fractional Brownian motion some convergence results obtained for a more general class of noise in \cite{Friz book,Friz paper}. The integral is understood within a rough path approach and the weak convergence is considered with respect to the $p$-variation topology.  
\medskip
 
The result is precisely the following: for $H\in \left(\frac 13, \frac 12\right]$ let us consider the solution $X^H$ of \eqref{eq:general} where $Y\equiv W^H$, which defines a probability distribution on the space $C^{1/3}([0,T])$ of $\frac 13$-H\"older continuous functions. We show that whenever $H\to H_{\infty}\in \left(\frac 13, \frac 12\right]$, it holds that $X^{H}\xrightarrow{d}X^{H_{\infty}}$, where $\xrightarrow{d}$ denotes the convergence in distribution on $C^{1/3}([0,T])$.

\smallskip

The proof   relies on the observation that the  solution operator which maps the lift $(W^H,\mathbb{W}^H)$ of  the noise $W^H$ into the solution $X^H$  can be made continuous (\cite{Hairer book,Friz book} ), so that 
 it is sufficient to show that $(W^H,\mathbb{B}^H)\to (B^{H_{\infty}},\mathbb{B}^{H_{\infty}})$ and to exploit the continuity of the solution map to deduce that $X^H\to X^{H_\infty}$. In the present work, we prove the continuity of the lift by following the standard scheme, that is by first  establishing   the tightness property and then by  identifying the limit. Our main contribution is  the proof of the tightness  in the specific case of the fractional Brownian motion taking advantage of some fundamental results of the rough path theory (\cite{Friz book}).

\section{H\"older spaces and lifted paths spaces}
Let us recall now the main functional spaces useful within the rough path theory.

\begin{definition}\label{def:mathcal_C_alpha}
Let $\alpha>0$. Given a Banach space $(E,|\cdot|_E)$, a function $Y:[0,T]\rightarrow E$  is a \emph{$\alpha$ -H\"older continuous} function if the seminorm
 \begin{equation}
\|Y \|_{\alpha} := \mathop{\sup}_{t\not= s}\frac{|Y_t-Y_s|_E}{|t-s|^{\alpha}}
\label{eq:hoelder_continuity}
 \end{equation}
 is finite.
    Let $\mathcal{C}^{\alpha}([0,T];E)$ be the space of all     $\alpha$ -H\"older continuous functions from $[0,T]$ into $E$.
    A norm on $\mathcal{C}^{\alpha}$ is define as follows
     \begin{equation}
     \|Y\|_{\mathcal{C}^{\alpha}}  = |Y_0|_E+ \|Y \|_{\alpha} .
     \label{eq:C_alpha_norm}
     \end{equation}
\end{definition}
We may extend the space introduced in Definition \ref{def:mathcal_C_alpha} to the functions defined on $[0,T]^2$.
\begin{definition}\label{def:mathcal_C_alpha_2}
    Let $\alpha>0$. Given  a Banach space $(E,|\cdot|_E)$, we define the space $\mathcal{C}^{\alpha}_2$ as the set of functions   $W:[0,T]^2\rightarrow E$  such that the seminorm
    \begin{equation}
    \|W \|_{\mathcal{C}_2^\alpha} := \mathop{\sup}_{t\not= s}\frac{|W(s,t)|_E}{|t-s|^{\alpha}} 
    \label{eq:hoelder_continuity_2}
    \end{equation}
    is finite.
\end{definition}

\begin{definition}\label{def:mathcal_C_alpha+mathcal_C_alpha_2}
    Let $\alpha>0$. Given  a Banach space $(E,|\cdot|_E)$, the  vector space $\mathcal{C}^{\alpha}\oplus \mathcal{C}^{2\alpha}_2$ is the set of the pair functions $(X,W)$  with $X:[0,T]\rightarrow E$  and $W:[0,T]^2\rightarrow E$,   endowed by the norm
    \begin{equation}
    \|(X,W) \|_{\mathcal{C}^{\alpha}\oplus \mathcal{C}^{2\alpha}_2} :=  \|X\|_{\mathcal{C}^{\alpha}} +  \|W \|_{\mathcal{C}_2^{2\alpha}}.
    \label{eq:norm_C_alpha+C_2alpha}
    \end{equation}
Such a space is a Banach space.
\end{definition}

The rough path may be seen  as a subspace of the Banach space given in Definition \ref{def:mathcal_C_alpha+mathcal_C_alpha_2}.

\begin{definition}\label{def:rough_path}

    Let $\alpha \in \left(\frac{1}{3},\frac{1}{2}\right]$.
An \emph{$\alpha$-H\"older   rough path}
    is a pair of functions
$(X,\mathbb X)\in \mathcal{C}^{\alpha}([0,T],E)\oplus\mathcal{C}_2^{2\alpha}([0,T]^2,E) $
    such that
    the so called Chen's relation is satisfied, i.e. for any $s,u,t \in [0,T]$,
    \begin{equation}\label{eq:chen}
    \mathbb X_{s,t}-\mathbb X_{s,u}-\mathbb X_{u,t}  =(X_u-X_s)\otimes (X_t-X_u).
    \end{equation}
    We denote by    $  \mathscr{C}^{\alpha}$ the subspace of $\mathcal{C}^{\alpha}([0,T],E)\oplus\mathcal{C}_2^{2\alpha}([0,T]^2,E)$ such that the Chen's relation
    \eqref{eq:chen} is satisfied, endowed by the distance
    \begin{eqnarray}\label{eq:distance_rough_path}
    \rho_{\alpha}\left((X,\mathbb X),(Y,\mathbb Y)\right)&=&\|X-Y \|_\alpha+ \|\mathbb X-\mathbb Y \|_{\mathcal{C}^{2\alpha}_2}
    \\&=& \nonumber\mathop{\sup}_{t\not= s}\frac{|X_t-X_s-(Y_t-Y_s)|_E}{|t-s|^{\alpha}} +\mathop{\sup}_{t\not= s}\frac{|\mathbb X_{{s,t}}-\mathbb Y_{{s,t}}|_E}{|t-s|^{2\alpha}}.
    \end{eqnarray}

\end{definition}

\medskip

\begin{remark}
    \label{rem: C^alpha is not a vector space}
    The space $\mathscr{C}^\alpha$ is a subset of the vector space $\mathcal{C}^\alpha \oplus \mathcal{C}^{2\alpha}_2$, but it is not a linear subspace, due to the non-linear scaling given by \eqref{eq:chen}. In detail, for $(X,\mathbb X)\in \mathscr{C}^\alpha$ and $\lambda\in \mathbb R$ we have that
\begin{eqnarray*}
    \lambda \mathbb X_{s,t}-\lambda \mathbb X_{s,u}-\lambda \mathbb X_{u,t}&=& \lambda( \mathbb X_{s,t}- \mathbb X_{s,u}-\mathbb X_{u,t})\\
    &=& \lambda((X_u-X_s)\otimes (X_t-X_u))\\
    &\neq& \lambda^2 ((X_u-X_s)\otimes (X_t-X_u))\\
    &=& (\lambda X_t-\lambda X_u) \otimes (\lambda X_u-\lambda X_s).
    \end{eqnarray*}
    Hence, the Chen's relation is not satisfied by $\lambda(X,\mathbb X)$, except for $\lambda=0,1$. On the contrary,  if $(X,\mathbb X)\in \mathscr{C}^\alpha$, then $(\lambda X, \lambda^2 \mathbb X)$ satisfied the Chen relation \eqref{eq:chen}.
\end{remark}
 The non-linear scaling property  given by $(X,\mathbb X)\to (\lambda X, \lambda^2 \mathbb X)$ suggests the definition of the following quantity, which is homogeneous with respect to \eqref{eq:chen}.
 \begin{definition}
    We define on $\mathscr{C}^\alpha$ the $\alpha$\textit{-H\"older rough path norm} as the quantity given by
    \begin{equation}
    \label{eq: rough path homogeneous norm}
    ||(X,\mathbb X)||_{\mathscr{C}^\alpha}:=||X||_\alpha+ \sqrt{||\mathbb X||_{C_2^{2\alpha}}}.
    \end{equation}
 \end{definition}
 \begin{remark}
    The quantity $  ||(X,\mathbb X)||_{\mathscr{C}^\alpha}$ is not a norm in the usual sense, because $||\lambda (X,\mathbb X)||_{\mathscr{C}^\alpha}\neq|\lambda|\cdot ||(X,\mathbb X)||_{\mathscr{C}^\alpha}$, but it scales correctly with respect to  \eqref{eq:chen} by preserving the transformation $(X,\mathbb X)\to (\lambda X, \lambda^2 \mathbb X)$. Indeed,   we have that
    $$||(\lambda X, \lambda^2 \mathbb X)||_{\mathscr{C}^\alpha} =|\lambda|\cdot ||(  X,  \mathbb X)||_{\mathscr{C}^\alpha}. $$
 \end{remark}

Let us observe that neither \eqref{eq:chen} nor the definition of $\mathscr{C}^\alpha$ imply any type of chain rule or integration by parts formula.

\begin{definition}
    \label{def: geometric rough paths space}
    Let $E=\mathbb R$. We define the space $\mathscr{C}^\alpha_g$ of \textit{geometric rough paths} as the space of rough paths  in $\mathscr{C}^\alpha$ which moreover satisfy the following condition
    \begin{equation}
    \label{eq: geometric rough path condition}
    \mathbb X_{s,t}= \frac{1}{2}(X_t-X_s)^2.
    \end{equation}
\end{definition}
\begin{remark}
    We note that  in the case $E= \mathbb R^d, d=1$, the geometric rough path condition \eqref{eq: geometric rough path condition} completely determines the form of $\mathbb X$,.   If we consider paths with values in $\mathbb R^d$, the function $\mathbb X$ becomes $\mathbb R^d \otimes \mathbb R^d $-valued (matrix-valued) and condition \eqref{eq: geometric rough path condition} becomes $\text{Sym}(\mathbb X)=\frac{1}{2}(X_{s,t}\otimes X_{s,t})$.   We refer to \cite{Hairer book} for a precise description of the multidimensional case $\mathbb R^d, d>1$.
\end{remark}

\subsection{Gaussian processes as rough paths}

\label{subsec: Gaussian rough paths}

 We consider a canonical rough path structure
for a class of continuous Gaussian processes,  which satisfy a
specific condition upon the covariance structure. In order to
define the properties upon the covariance, one needs to introduce
the right definition of variation (\cite{Friz book,Hairer book}).
The
fractional Brownian motion belongs to such a class of processes.
\smallskip

\begin{definition}
    \label{def: control}
    Let $\Delta_{T}:=\{0\leq s \leq t \leq T\}$ and consider a map  $\omega:\Delta_{T}\times\Delta_{T} \to [0,\infty)$.
    We say that $\omega$ is a \textit{2D control} if it is super-additive in the following way:
    given a rectangle $R\subset [0,T]^2$ and any finite partition $\{R_j,1\leq j\leq n\}$ of $R$, we have
    $$\omega(R)\geq \sum_{j\leq n} \omega(R_j).$$
    Given a function $f$ defined on rectangles, we say that $f$ is \textit{controlled by}
    the control $\omega$ if, for any rectangle $R\subset [0,T]^2$,
    the following estimate holds
    $$|f(R)|\leq \omega(R).$$
\end{definition}
\begin{definition}\label{def:rectangular_increment}
Given a function $f:[0,T]^2\to\mathbb R$, we denote by
$$R:=\binom{s,t}{u,v}:=[s,t]\times[u,v]$$
a rectangle of $[0,T]^2$ and
$$f \binom{s, t}{u , v}:=f(t,v)-f(t,u)-f(s,v)+f(s,u) $$
the \emph{rectangular increment} of $f$
where  $0\leq s\leq t\leq T$ and $0\leq u\leq v \leq T.$ 
\end{definition}
\begin{definition}\label{def:partitions}
 Let us denote by  $R_{i,j}=(t_i,t_{i+1}]\times(t^\prime_j,t^\prime_{j+1}] \subseteq R$, such that $\{t_i \}_i \in \mathcal{D}([s,t]), \{t^\prime_j \}_j \in \mathcal{D}([u,v])$, 	where $\mathcal{D}([s,t])$ is the family of partitions of the interval $[s,t]$ and $\pi(R)$ a (generic) partition of $R$.

In the following we might denote a partition $\widetilde{\pi}(R)=\{t_i,t^\prime_j\}_{i,j}\in 	\mathcal{D}^2(R)$, i.e. 
$$
\{t_i,t^\prime_j\}_{i,j}:=\{R_{i,j}\}_{i,j}=\widetilde{\pi}(R).
$$
	
	The set $\mathcal{D}^2(R)$ is the family of  \emph{regular or grid-like partitions} of  $R$,  
	\begin{eqnarray}
	\mathcal{D}^2(R)&=&\left\{ \widetilde\pi(R)=\{R_{i,j}\}_{i,j}  : \bigcup_{i,j \in\mathbb N} R_{i,j}=R \,\,  \right\}
	\\
	&=& \mathcal{D}([s,t]) \times \mathcal{D}([u,v]) \nonumber
	\end{eqnarray}

The set $\mathcal{P}(R)$ denotes the family of all rectangular partitions or tessellations of $R$, i.e. all families $\pi$  such that
$$
\mathcal{P}(R)=\left\{ \pi(R)=\{R_j\}_{j\in\mathbb N}: R_j\not= \emptyset; \mathring{R}_i\cap \mathring{R}_i=\emptyset, i\not= j; \bigcup_{j\in \mathbb N} R_j=R\right\}.
$$
Since not any partition is of grid-like type, one has trivially that, for any $R\subseteq [0,T]^2$,
\begin{equation}
\label{eq:partition_D2_subset_P}
\mathcal{D}^2(R)\subset\mathcal{P}(R).
\end{equation}

\end{definition}
  
\begin{definition}
    Let $f:[0,T]^2\to \mathbb R$ and $p\in [1,\infty)$. For any rectangle  $R \subset [0,T]^2$  the following quantity
    \begin{equation}
    \label{eq: definition rho-variation}
    V_p(f,R):= \Bigg(\sup_{ \{t_i,t^\prime_j\}_{i,j}\in 	\mathcal{D}^2(R)   }   \sum_{i,j} \Bigg|f\binom{t_i,t_{i+1} }{t_j, t_{j+1}}\Bigg|^p \Bigg)^{\frac{1}{p}}
    \end{equation}
    is called the \textit{$p$-variation of $f$ over $R\subseteq [0,T]^2$}.
The function $f$ has \textit{finite $p$-variation} if it holds that $$V_p(f,[0,T]^2)<\infty.$$
\end{definition}

\begin{definition}\label{def:controlled_rho_variation}
    Let $f:[0,T]^2\to \mathbb R$ and let $p\in [1,\infty)$. For any rectangle $R\subset[0,T]^2$ we
    define the \textit{controlled $p$-variation} as
    \begin{equation}
    \label{eq: definition controlled rho-variation}
    |f|_{p\text{-var},R}:= \Bigg(\sup_{\pi\in \mathcal{P}(R)}
    \sum_{A\in \pi}\Big| f(A) \Big|^p \Bigg)^{\frac{1}{p}}.
    \end{equation}
\end{definition}
\begin{remark}
	Recall that if $x:[0,T]\to \mathbb R$,  then for any $[u,v]\subset[0,T]$, and any $p>0$ $V_p(x,[u,v])=\left|x\right|_{p-var,[u,v]}$. Furthermore, defining $\omega_1$,  for any $s,t \in [0,T]$, as $\omega_1(s,t)=\left|x\right|_{p-var,[u,v]}^p$, we have that $\omega_1$ is a 1D control and it controls $x$, i.e. 
	\begin{eqnarray}
	\omega_1(s,u)+\omega_1(u,t)&\le& \omega_1(s,t);\label{eq:omega_1_control}\\
	\big|x(t)-x(s)\big| &\le& \omega_1(s,t)^{1/p}.\label{eq:omega_1_controls_x}
	\end{eqnarray}
\end{remark}
\begin{remark}
    Since  for any $f$ and any $R\subseteq [0,T]^2$ inclusion \eqref{eq:partition_D2_subset_P} holds,    for any $p\geq 1$ the following inequality holds
    \begin{equation}  V_p(f,R)\leq |f|_{p\text{-var},R}.
    \label{eq:rho_var_less_controlled_rho_var}
    \end{equation}
    Whenever $p>1$,  $$  V_p(f,R)< |f|_{p\text{-var},R}.$$
    We will see an example of this behavior in the case of the fractional Brownian motion $W^H$ in Proposition
    \ref{prop: p-variation fBm}.
\end{remark}

\medskip

Even if the $p$-variation and the controlled $p$-variation are different concepts, it is know they are $\varepsilon$-close concepts. Indeed, by means of the Young-Towghi's maximal inequality, in \cite{Friz paper} the authors prove the following result.

\begin{proposition}[\cite{Friz paper}, Theorem 1-4]
	Let $p\geq 1$ and  $\varepsilon>0$. There exists an explicit constant $C(p,\varepsilon)\geq 1$ such that for every $f:[0,T]^2\to \mathbb{R}$ and for every $R$ rectangle in $[0,T]^2$ it holds
	\begin{equation}
	\label{eq: epsilon close}
	\frac{1}{ C(p,\varepsilon)}|f|_{p+\varepsilon\text{-var},R}\leq V_{p}(f,R) \le |f|_{p\text{-var},R}.
	\end{equation}
Introducing $\alpha_p=  { p(p+ \varepsilon) }$, the constant $C$ is given by
\begin{equation}
\label{eq: C(p,eps)}
\begin{split}
C(p,\varepsilon)= \Big\{ \Big[ & 1+\zeta\Big(1+\frac{\varepsilon}{2\alpha_p+\varepsilon}\Big) \Big]^{1+\frac{\varepsilon}{2\alpha_p}}\times \zeta\Big(
1+\frac{\varepsilon}{2\alpha_p} \Big) 
+\Big[1+\zeta\Big(1+\frac{\varepsilon}{\alpha_p} \Big)\Big] \Big\},
\end{split}
\end{equation}
where $\zeta$ denotes the Riemann zeta function.

Furthermore, if $|f|_{p\text{-var},R}$ is finite, then it is superadditive as function of $R$.

\end{proposition}

\smallskip

\begin{remark}\label{remark:costant_C}
Note that, for any fixed $\varepsilon>0$,  $C(p,\varepsilon)$ is continuous as function of $p\in[1,\infty)$. Indeed,  since $\zeta(x)\to \infty$  when $x\to1^+$  it only diverges when $p\to \infty$.
\end{remark}

\smallskip

 \begin{remark}
 From \eqref{eq:rho_var_less_controlled_rho_var} and \eqref{eq: epsilon close} one obtain the following inequality
	\begin{equation}
\label{eq: inequality V}
V_{p+\varepsilon} (f,R) \leq C(p,\varepsilon)V_{p}(f,R).
\end{equation}
	\end{remark}

\medskip

Now let $\{X_t,t\in [0,T]\}$ be a real-valued centered continuous
Gaussian process with covariance structure given, for $s,t\in
[0,T]$, by
$$K(s,t):=\mathbb E[X_t X_s].$$

Given a continuous and centered Gaussian process $X$ with
covariance $K$, it is possible to construct a canonical rough path
$(X,\mathbb X)$, provided that the covariance function $K$ has
some $p$-variation regularity, and that the $p$-variation of $K$
is controlled by some 2D control $\omega$. We make it more
precise.

\begin{theorem}[\cite{Friz book}, Theorem 15.33]
    \label{teo: 15.33}
    Let $X_t$, for $t\in[0,T]$, be a centered continuous Gaussian process with values in $\mathbb{R}$. Suppose that there exists a $\rho\in[1,2)$ such that the covariance $K$ of $X$, given  2D control $\omega$ such that $\omega([0,T]^2) <\infty$
   \begin{equation}\label{eq:var_K_dominated_omega}
   |K|_{\rho\text{-var},R}\le \omega(R), \qquad \forall R\subseteq {[0,T]}^2,
   \end{equation}
    that is the covariance $K$ has finite controlled $\rho$-variation dominated by a 2D control $\omega$.

    Then, there exists a unique process $(X,\mathbb X)$ in $\mathscr{C}^\alpha$ such that $(X,\mathbb X)$ lifts $X$, in the sense that $\pi_1((X,\mathbb X)_t)=X_t-X_0$.
    Moreover, there exists a constant $C=C(\rho)$ such that for every $s\leq t$ and for every $q\geq 1$ it holds
    \begin{equation}
    \label{eq: estimate 15.33}  
     \E{\Big(|X_{s,t}|+|\mathbb X_{s,t}|^{1/2}\Big)^q}^{\frac{1}{q}}\leq C(\rho) \sqrt{q} \,\omega([s,t]^2)^{\frac{1}{2\rho}}.
    \end{equation}
    The lift $(X,\mathbb X)$ is unique and natural in the sense that it is the limit in the space of rough paths $\mathscr{C}^\alpha_g$ of any sequence $X_n$ of piecewise linear or mollified approximations to $X$ such that $||X_n-X||_\infty\to 0$ almost surely.
\end{theorem}

\begin{remark}
    Regarding the approximations to a rough path $(X,\mathbb X)$ via regular functions, we refer to  (\cite{Friz book}, Chapter 15), in which there is a large discussion about piecewise linear and mollified approximations of a Gaussian process. A complete discussion about this topic would exceed the scope of this work.
\end{remark}

\section{RDEs driven by a fractional Brownian motion}

The fractional Brownian motion with Hurst parameter $H \in [0,1]$ is a zero mean Gaussian process with covariance given by
\begin{equation}  \mathbb{E} [W^H_t W^H_s] = \frac{1}{2} (| t |^{2H} + | s |^{2H} - | t - s |^{2H})=:K^H(s,t).  \label{eq:covariance_fBm}
\end{equation}
\medskip

The parameter $H$ is responsible of the strength and the sign of the correlations between the increments. Indeed, for $H\in(0,1)\setminus\{\frac12\}$,  set $\widetilde{H}=H-\frac{1}{2}$, for any $t_1<t_2<t_3<t_4$, one may express the covariance of the increments in an integral form
\begin{equation}
\begin{split}\label{eq: correlation of increments in integral form}
E\Big[(W^H_{t_2}-W^H_{t_1})(W^H_{t_4}-W^H_{t_3})\Big]
&=2 \widetilde{H} H\int_{t_1}^{t_2}\int_{t_3}^{t_4}(u-v)^{2\widetilde{H}-1}du\,dv\\
&= \frac{1}{2}\Big( |t_4-t_1|^{2H}+|t_3-t_2|^{2H}\\
&\qquad \,\,\,\, -|t_4-t_2|^{2H}-|t_3-t_1|^{2H} \Big). 
\end{split}
\end{equation}

Since in the above integral form  the integrand is a positive function and $H>0$, it follows that the sign of the correlation depends only upon $\widetilde{H} $, being positive when $\widetilde{H}>0$, i.e. $H\in(\frac{1}{2},1)$, and negative when $\widetilde{H}<0$, i.e. $H\in(0,\frac{1}{2})$.

\medskip

For any $(s,t)\in [0,T]^2$, we denote by $|\pi_{s,t}|=\mathop{\max}_{j}|t_j-t_{j-1}|  $ the width of any partition $\pi_{s,t}\in \mathcal{D}([s,t])$. Then, given the process $\{W_t^H\}_{t\in \mathbb R_+}$, let  $\{ \mathbb{W}^H_{s, t}\}_{s,t}$ be
the iterated integral operator defined as
\begin{eqnarray}
\mathbb{W}^H_{s, t} &=&
\int_s^t (W^H_{\tau} - W^H_s) \circ d W^H_{\tau} :=
\frac{1}{2}\left(X_t - X_s\right)^2.
\end{eqnarray}

According with Definition \ref{def:rough_path},  the process $\mathbf{W}^H=(W^H,\mathbb W^H)$  is a geometric rough path.

\medskip
Let us now consider the RDE \eqref{eq:general} with driven signal given by a fractional Brownian motion $W^H$ of index $\frac{1}{3} < H < 1$.

\begin{equation}
  d X_t^{H} = \mu (X_t^H) d t +   \sigma(X_t^H)
 \circ d W^{H}_t.
  \label{eq: SDE fBm rp}
\end{equation}
The solution $X^H_t$ is interpreted in the following integral form
\begin{equation}
X^H_t  = X^H_0+ \int_0^t \mu (X^H_t ) d t +  \int_0^t  \sigma  (X^H_t) \circ d W^H_t, \label{eq:general integral form}
\end{equation}
with initial condition $X^H_0\in L^2(\Omega)$, where the stochastic integral is an integral with respect to a rough path $\mathbf{W}^H=(W^H,\mathbb W^H)$ defined over the process $W^H$. 
The following relevant results concerning existence, uniqueness and a notable continuity property of the above stochastic equation holds  (\cite{Friz book}).

\begin{theorem}
	\label{teo: continuity solution map}
	Given $H\in (\frac{1}{3},\frac{1}{2}]$, let $X^H_0=x_0\in \mathbb R$ be a constant and let $\mu,\sigma \in \mathcal{C}^3_b(\mathbb R)$ (three times differentiable bounded functions). Then there exists an unique solution  $X^H=(X^H_t)_{t\in[0,T]}$  to equation \eqref{eq: SDE fBm rp} with initial condition $x_0$. Moreover, the solution $X^H$ is a continuous function of  $\mathbf{W}^H=(W_t^H, \mathbb{W}^H_{s, t})$, in the sense that the solution map $S$, given by
	\begin{equation}
	\begin{split}
	S: \mathscr{C}^\alpha & \longrightarrow \mathcal{C}^{\alpha}([0,T]) \\
	\mathbf{W}^H & \longmapsto X^H,
	\end{split}
	\end{equation}
	is continuous, for any $0<\alpha<H$.
\end{theorem}

 \subsection{A weak continuity result with respect to the noise}

Here we provide the main results of the paper. Given the existence and uniqueness result of a solution $X^H$ to \eqref{eq: SDE fBm rp} stated  by Theorem \ref{teo: continuity solution map}, a natural question that can be addressed is about the continuity of such a solution $X^H$   with respect to the parameter $H$.

\medskip

In  order to use the rough paths techniques in a non-trivial way we restrict to the most interesting case of  $W^H$,  fractional Brownian motion of Hurst parameter $H\in(\frac{1}{3},\frac{1}{2}]$. Indeed, when $H>\frac{1}{2}$ the regularity of the noise allows for a classical solution theory in the sense of Young integration. By Theorem \ref{teo: continuity solution map} we have that a solution to \eqref{eq: SDE fBm rp} exists and it is unique and, moreover, the solution operator is continuous from $\mathscr{C}^{\alpha}$ to $\mathcal{C}^\alpha([0,T])$, for any $0<\alpha<H$.
	When $H=\frac{1}{2}$, the solution $X^{\frac{1}{2}}$ to \eqref{eq: SDE fBm rp} becomes a Stratonovich solution of an SDE driven by a standard Brownian motion  (sBm). This is a direct consequence of the well-known fact that when we lift a sBm $W^{\frac{1}{2}}$ to a geometric rough path, one obtains the Stratonovich integral. 

\medskip

We first consider some boundedness results upon the covariance $K^H$ given by
\eqref{eq:covariance_fBm} with respect to the $p$-variation and the controlled $p$-variation. We see how their behavior may be completely different, as first point out in \cite{Friz paper}. In the following, in order to prove the uniformity of the involved estimates, we  always make explicit  the dependence upon the parameter $H$. 

\begin{proposition} 	\label{prop: p-variation fBm}
	The covariance $K^H$ of a a fractional Brownian motion of parameter $H\in\left(0,\frac{1}{2}\right]$, given by \eqref{eq:covariance_fBm}, has bounded $\frac{1}{2H}$-variation $V_{\frac{1}{2H}}\left(K^H,[0,T]^2\right)$, which, moreover, for any $s<t$ satisfies 
	\begin{equation}
	\label{eq: bound V_1/2H fBm}
	V_{\frac{1}{2H}}(K^H,[s,t]^2)\leq 3 |t-s|^{2H}.
	\end{equation}
	Moreover, one has that the controlled $\frac{1}{2H}$-variation is infinite, that is, for any $R\subset [0,T]^2$
		\begin{equation}
	\label{eq: bound controlled_1/2H fBm}
	\left|K^H\right|_{\frac{1}{2H}\text{-var,}R}=\infty.
	\end{equation}
\end{proposition}

\begin{proof}
 Let us  prove the inequality \eqref{eq: bound V_1/2H fBm}.  For the proof of the unboundedness of the controlled $1\backslash 2H$-variation, we refer to (\cite{Friz paper}, Example 2).
 
 \medskip
 
 Without loss of generality, we consider $T=1$. For any $u_1,u_2\in[0,1]$,  with $u_1\le u_2$, we use the symbol   $W^H_{u_1,u_2}$ for the increment $W^H_{u_2}-W^H_{u_1}$.  Furthermore we put $p=1/ 2H \ge 1$. 
 
 \medskip
 
	Fixing the interval $[s,t]\subset [0,1]$, let us take two partitions $\pi_1=\{t_i\}_i,\pi_2=\{t'_j\}_j\in\mathcal{D}([s,t])$, with $n_j= card(\pi_j), j=1,2$.  For any fixed $\{t_i,t_{i+1}\}\in \pi_1$, we consider the function $f^i:=\mathbb E\left[W^H_{t_i,t_{i+1}}, W^H_.\right]$. Since  the 1D $p$-variation $\left|f^i\right|^p_{p-var,[ t_j^\prime,t_{j+1}^\prime]}$ is a control, by \eqref{eq:omega_1_controls_x} we obtain
	$$
	f^i(t_{j+1}^\prime)-f^i(t_{j}^\prime)\le \left[\left|f^i\right|^p_{p-var,[ t_j^\prime,t_{j+1}^\prime]}\right]^{\frac{1}{p}},
	$$
	and, by super additivity of the control,
		\begin{eqnarray}\label{eq:prima_prop_1}
		\sum_{j=1}^{n_2} \Big| 	f^i(t_{j+1}^\prime)-f^i(t_{j}^\prime) \Big|^{p}&\leq & \sum_{j=1}^{n_2}  \left|f^i\right|_{p-var,[ t_j^\prime,t_{j+1}^\prime]}^p\leq  \left|f^i\right|_{p-var,[ s,t]}^p.
	\end{eqnarray}
	By considering the subpartition of $[s,t]$ given by $[s,t_i],[ t_i, t_{i+1}],[t_{i+1},t]\subset [s,t]$, since when $p\ge 1$ it holds that $(a+b+c)^p\le 3^{p-1}(a^p+b^p+c^p)$, we have
	\begin{eqnarray}\label{eq:prima_prop_2}
	\left|f^i\right|_{p-var,[ s,t]}^p&\le& 3^{p-1}   \left(\left|f^i\right|_{p-var,[ s,t_i]}^p +\left|f^i\right|_{p-var,[ t_i,t_{i+1}]}^p+\left|f^i\right|_{p-var,[ t_{i+1},t]}^p\right).
\end{eqnarray}
	Now we need to estimate the three $p$-variations on the right hand side. First, we observe that the first and third terms are the $p$-variations of the covariance of the increments of disjoint time increments, despite the second one.
	
	For the estimation of the latter, we notice that, whenever $[u,v]\subset[t_i, t_{i+1}]$,
with $0\leq s\leq u\leq v\leq t\leq 1$, we obtain 
	\begin{eqnarray*}
	\Big| \mathbb E  \left[   W^H_{t_i,t_{i+1}}W^H_{u,v} \right]\Big| & =& \Big|\mathbb E  \left[  (W^H_{t_{i+1}}-W^H_{t_i})(W^H_v-W^H_u)\right]\Big| \\
	& =& \Big| \mathbb E  \left[ (W^H_{t_{i+1}}-W^H_v+W^H_v-W^H_u+W^H_u-W^H_{t_i})(W^H_v-W^H_u)\right] \Big| \\
	& =&\Big| \mathbb E  \left[ (W^H_{t_{i+1}}-W^H_v)(W^H_v-W^H_u)\right] + \mathbb E  \left[ (W^H_v-W^H_u)^2\right] \\
	&&  \quad +\mathbb E  \left[ (W^H_u-W^H_{t_i})(W^H_v-W^H_u)\right]\Big| \\
		& =& \frac{1}{2}\Big|  |t_{i+1}-u|^{2H}-|t_{i+1}-v|^{2H} + |v-t_i|^{2H}-|u- {t_i}|^{2H}  \Big| \\
	& \leq & |u-v|^{2H}= |u-v|^{\frac{1}{p}}.
	\end{eqnarray*}
	The last inequality is due to the fact that, since $0< 2H \leq 1$, if $h_1\le h_2 \le h_3 $, $|h_3-h_1|^{2H}=|h_3-h_2+h_2-h_1|^{2H}\leq |h_3-h_2|^{2H}+|h_2-h_1|^{2H}$, i.e. $|h_3-h_1|^{2H} -|h_3-h_2|^{2H}\leq |h_2-h_1|^{2H}$. \\ 
As a consequence by definition of $p-$ variation
\begin{equation}
\label{eq:prima_prop_termine_intermedio}
\left|f^i\right|_{p-var,[ t_i,t_{i+1}]}^p=\sup_{\{u_j\}_j\in \mathcal{D}\left([t_i,t_{i+1}]\right)}\sum_j \Big| \mathbb E  \left[   W^H_{t_i,t_{i+1}}W^H_{u_j,u_{j+1}} \right]\Big|^p\le \left|t_{i+1}-t_i\right|.
\end{equation}

From \eqref{eq: correlation of increments in integral form} it is clear that in the case  $H \le\frac{1}{2}$ the disjoint increments of the fractional Brownian motion have negative correlations. This implies  

\begin{eqnarray*}
	\label{eq: estimate3 Example 1 Friz}
\Big |\mathbb E  \left[  W^H_{t_i,t_{i+1}}  W^H_{\cdot }\right] \Big|^{p}_{p-var;[s,t_{i}]}  &\leq&
\sup_{\{u_j\}_j\in \mathcal{D}\left([s,t_{i}]\right)}\left|  \mathbb E  \left[ \sum_{j} W^H_{t_i,t_{i+1}}  W^H_{_{u_j,u_{j+1}}}\right]  \right |^{p}  \\&=&
 \Big|  \mathbb E  \left[  W^H_{t_i,t_{i+1}}  W^H_{s,t_{i}}\right]  \Big |^{p}.  \\
\end{eqnarray*}

Again by \eqref{eq: correlation of increments in integral form} 
\begin{eqnarray*}
 \mathbb E  \left[  W^H_{t_i,t_{i+1}}  W^H_{s,t_{i}}\right] &=& \frac{1}{2}\Big( |t_{i+1}-s|^{2H}  -|t_{i+1}-t_i|^{2H}-|t_i-s|^{2H} \Big) \\
 &\le& \frac{1}{2}\left( |t_{i+1}-s|^{2H} -|t_i-s|^{2H}\right)\le \frac{1}{2}  |t_{i+1}-t_i|^{2H}\\
 &\le& |t_{i+1}-t_i|^{2H}=|t_{i+1}-t_i|^{\frac{1}{p}}.
\end{eqnarray*}
 As a consequence
 \begin{equation}
 \label{eq:prima_prop_primo_termine}
 \left|f^i\right|_{p-var,[ s,t_{i}]}^p=\sup_{\{u_j\}_j\in \mathcal{D}\left([s,t_{i}]\right)}\sum_j \Big| \mathbb E  \left[   W^H_{t_i,t_{i+1}}W^H_{u_j,u_{j+1}} \right]\Big|^p\le \left|t_{i+1}-t_i\right|.
 \end{equation}
 In the same way, one con prove that
 \begin{equation}
\label{eq:prima_prop_terzo_termine}
\left|f^i\right|_{p-var,[t_{i+1},t]}^p=\sup_{\{u_j\}_j\in \mathcal{D}\left([t_{i+1},t]\right)}\sum_j \Big| \mathbb E  \left[   W^H_{t_i,t_{i+1}}W^H_{u_j,u_{j+1}} \right]\Big|^p\le \left|t_{i+1}-t_i\right|.
\end{equation}	
	\smallskip
	Finally, by \eqref{eq: definition rho-variation},  \eqref{eq:prima_prop_1} and \eqref{eq:prima_prop_2}, one obtains 
	\begin{eqnarray*}
	 			\left(	V_{p}(K^H,[s,t]^2) \right)^p &=& \sup_{ \{t_i,t^\prime_j\}_{i,j}\in 	\mathcal{D}^2(R)   }   \sum_{i,j} \Bigg|\E{  W^H_{t_i,t_{i+1}}  W^H_{t_j,t_{j+1}}  }\Bigg|^p \\
	 			&\le&  \sup_{i}  \sum_i \left|f^i\right|_{p-var,[ s,t]}^p \\
	 			&\le&  3^{p-1} \sup_{\{t_i\}\in \mathcal{D}([s,t])}  \sum_i  \left(3 \left|t_{i+1}-t_i\right|\right)=3^p |t-s|, 
	\end{eqnarray*}
i.e.
 	\begin{eqnarray*}
 	V_{p}\left(K^H,[s,t]^2\right)   
 	&\le&  3  |t-s|^{\frac{1}{p}}= 3  |t-s|^{2H}.
 \end{eqnarray*}

\end{proof}

\medskip
  
The following theorem states the main result of the present work, that is the continuity of the solution of the equation \eqref{eq: SDE fBm rp} with respect to the Hurst parameter in the case  $H \in \left( 1/3,1/2\right]$.

\begin{theorem}	
	\label{th: Weak continuity RDE}
	
	Let $\{H_n\}_{n\in\mathbb N}\in \left( 1/3,1/2\right]^{\mathbb N}$ be a sequence of Hurst parameters and let $H_{\infty} \in \left(1/3, 1/2 \right]$. Let   $\{W^{H_n}\}_{n\in H}$  be a family of fractional Brownian motions, each of them independent upon a random variable $X_0\in L^2(\Omega)$. 
For any $n\in \mathbb N,$ let us denote by $X_n^{X_0}$ the solution to the equation \eqref{eq: SDE fBm rp} with $H = H_n$, for $t\in[0,1]$ and with initial condition $X_0$. Suppose that $\mu, \sigma \in C_b^3(\mathbb R)$. 
 Then, the sequence $\left\{X_n^{X_0}\right\}_{n\in \mathbb N}$ converges to $X_{\infty}^{X_0} $ in distribution in the space $C^{\frac{1}{3}} ([0, 1])$.
\end{theorem}

\begin{proof}  Given the continuity of the solution map stated in Theorem \ref{teo: continuity solution map}, we have only to show that   
\begin{equation}\label{eq:convergence_W_lifted}
\mathbf{W}^{H_n}=(W_\tau^{H_n},\mathbb{W}^{H_n}_{s,t}) \xrightarrow{n\to \infty}\mathbf{W}^{H_\infty}= (W_\tau^{H_\infty},\mathbb{W}^{H_\infty}_{s,t}),
\end{equation} 
	in $\mathscr{C}^{\frac{1}{3}}([0,1])$. 
	
	\bigskip

	 The first step is to prove the tightness of $\left\{\mathbf{W}^{H_n}\right\}_n$. The thesis follows by the Kolmogorov-Lamberti criterion   (see  \cite{Friz book}, Corollary A.11),  if we establish that  there exist constants $M>0, q>r> 1$ such that $\frac{1}{r}-\frac{1}{q}>\frac{1}{3}$ and, moreover, that the following estimate holds
	 \begin{equation}
	 \label{eq:kolmogorov_lamperti_criterion}
	 \sup_{n\in \mathbb{N}}\E{d(\mathbf{W}^{H_n}_t,\mathbf{W}^{H_n}_s)^q}^{\frac{1}{q}} \le M \left|t-1\right|^{1/r}.
	 \end{equation}
\medskip

Let us denote  for any $n\in \mathbb N$   $p_n= \frac{1}{2 H_n} \in [1, 3/2). $  Using the hypothesis  $H_n \rightarrow H_{\infty}>\frac{1}{3}$, there exist a $\delta>0$ and an $n_0(\delta)\in \mathbb N$ such that $H_n>\frac{1}{3}+\delta$  for any $n>n(\delta)$. Defining the following constants
\begin{equation}
\label{eq:rho_theo}
 \rho:=\sup_{n\geq n_0(\delta)} p_n <\frac{3}{2},  \quad \varepsilon_n=\rho+\epsilon -p_n >0,
\end{equation}
where $0<\varepsilon<\frac{3}{2}-\frac{3}{2(1+3\delta)}$ is some fixed real number, 
for any $n\in \mathbb N$ and  any $R \in [0,T]^2$, by  \eqref{eq: epsilon close} it holds that
\begin{equation}\label{eq:rho+epsilon_var_finite}
\begin{split}
\left|K^{H_n}\right|_{\rho+\varepsilon\text{-var},R}=
\left|K^{H_n}\right|_{p_n+\varepsilon_n\text{-var},R}&\leq C\Big(p_n,\varepsilon_n\Big) V_{p_n}(K^{H_n},R) \\
&\leq \overline{C}  V_{p_n}(K^{H_n},R) < \infty.
\end{split}
\end{equation} 
where $\bar{C}=\sup_{n>n_0(\delta)}C\Big(p_n,\varepsilon_n\Big)$, which is finite since $p_n<\rho$, $\epsilon_n>\epsilon$
and by Remark \ref{remark:costant_C} $C(p,\tilde{\epsilon})$,  is a continuous function for $p<\rho$ and $\tilde{\epsilon}>\epsilon>0$. Furthermore the boundedness of the $p_n$-variation $ V_{p_n}(K^{H_n},R)$ is guaranteed by Proposition \ref{prop: p-variation fBm}.

 \medskip

For any $n>n(\delta)$  the  2D control $$\omega_{H_n}:=|K^{H_n}|^{\rho+ {\varepsilon}}_{\rho+\varepsilon\text{-var},R}.$$ 
 is an H\"older dominated control, uniformly in $n>n(\delta)$. Indeed   by \eqref{eq:rho+epsilon_var_finite} 
\begin{equation}
\label{eq: holder bound omega}
\begin{split}
\omega_{H_n}([s,t]^2)&\leq \overline{C}^{\rho+\varepsilon}  V_{p_n}\left(K^{H_n},[s,t]^2\right)^{\rho+\varepsilon}\\&\le  \left(3 \overline{C}\right)^{\rho+\varepsilon} \left| t-s \right| ^{\frac{\rho+\varepsilon}{p_n}}\\ &\le C|t-s|.
\end{split}
\end{equation}
The last two inequalities are due to  Proposition  \ref{prop: p-variation fBm}, to the fact   that   $ (\rho+{\varepsilon})\backslash p_n>1$ together with the assumption  $|t-s|\leq 1$. \\
In particular we obtain that
\begin{equation}
\label{eq: holder bound omega_01}
\begin{split}
\omega_{H_n}([0,T]^2)&\leq M_1.
\end{split}
\end{equation}

Moreover, since for any $n> n(\delta)$, $|K^{H_n}|_{(\rho+\epsilon)\text{-var},R}\le \omega_{H_n}(R) $,  condition \eqref{eq:var_K_dominated_omega} in  Theorem  \ref{teo: 15.33} is satisfied and so by \eqref{eq: estimate 15.33} and \eqref{eq: holder bound omega} we obtain that there exists a constant $\widetilde{C}=\widetilde{C}(\rho+\epsilon)$ such that for every $q\in [1,\infty)$ and for every $s,t\in[0,1]$
	\begin{equation} \label{eq:supE_inequality}
	\begin{split}
	\sup_{n\in \mathbb{N}} \E{ d(\mathbf{W}^{H_n}_t,\mathbf{W}^{H_n}_s)^q}^{\frac{1}{q}}
	\leq & \sup_{n\in \mathbb{N}} \left[\widetilde{C} \sqrt{q} \,\, \omega_{H_n}\left([s,t]^2\right)^{\frac{1}{(\rho+\epsilon)}} \right]\\
	\leq &  \widetilde{C}\sqrt{q} \sup_{n\in \mathbb{N}} |t-s|^{\frac{1}{(\rho+\epsilon)}}.
	\end{split}
	\end{equation}
	Since there exists $\epsilon$ such that ${\rho}+ \epsilon \in (1,3/2)$, then $\frac{1}{r}=\frac{1}{(\rho+\epsilon)} \in (1/3,1/2).$ As a consequence, by choosing $q>>1$ such that  $\frac{1}{r}-\frac{1}{q}>\frac{1}{3}$ and defining $M=\widetilde{C}\sqrt{q}$, from \eqref{eq:supE_inequality} we finally get the inequality \eqref{eq:kolmogorov_lamperti_criterion}.  Kolmogorov-Lamperti tightness criterion implies that the sequence $\mathbf{W}^{H_n}$ is tight in $\mathscr{C}^{\frac{1}{3}}$, and thus it possesses a subsequence converging to some limit $\mathbf{Y}$.
	
	\medskip
	
The last step is in the identification of the limit $\mathbf{Y}$ as $\mathbf{W}^{H_\infty}$. 
\medskip

Once the tightness has been achieved, we only need to show that the finitely dimensional distributions of $\mathbf{W}^{H_n}$ converge to the ones of $\mathbf{W}^{H_\infty}$, i.e. for any $m,\ell \in \mathbb N $, and for any increasing times $\{t_j\}_{j=1}^m,$  $\{s_i\}_{i=1}^\ell$, $\{k_i\}_{i=1}^\ell$, such that,  for any $i =1,\ldots,\ell$,   $s_i \leq k_i$,   we have that  $\left(W^{H_n}_{t_1},\dots,W^{H_n}_{t_m}, \right.$ $\left.\mathbb{W}^{H_n}_{s_1,k_1},\dots,\mathbb{W}^{H_n}_{s_{\ell},k_{\ell}}\right)$ converges in probability to  $\left(W^{H_{\infty}}_{t_1},\dots,W^{H_{\infty}}_{t_m}, \mathbb{W}^{H_{\infty}}_{s_1,k_1},\right.$ $\left.\dots,\mathbb{W}^{H_{\infty}}_{s_{\ell},k_{\ell}}\right)$. 

	First of all we note that since $W^n=\{W^{H_n}\}_n$ is a Gaussian process for any $n\in \mathbb N$ and the covariance of $W^n$ converges to the covariance of $W^{\infty}$, the finitely dimensional distributions $\left(W^{H_n}_{t_1},\ldots,W^{H_n}_{t_m}\right)$ converge in probability to $\left(W^{H_{\infty}}_{t_1},\ldots, W^{H_{\infty}}_{t_m}\right)$.
As far as concern the convergence of $\left(\mathbb{W}^{H_n}_{s_1,k_1},\dots,\mathbb{W}^{H_n}_{s_{\ell},k_{\ell}}\right)$  we  prove the one dimensional case, i.e. that, given $s \leq k \in [0,1]$,  the random variable $\mathbb{W}^{H_n}_{s,k}$ converges to $\mathbb{W}^{H_{\infty}}_{s,k}$. The general case is a  straightforward generalization. 
	
	\bigskip
	
 Let $D$ be a partition of $[0,1]$ having width $\delta_D$ and let $S_3\left(W^{H_n},D\right)_{s,k}$ be the piecewise linear approximations of $\mathbb{W}_{s,k}^{H_n}$. By \eqref{eq: holder bound omega_01} we obtain that $\sup_{n\in\mathbb{N}}\omega_{H_n}([0,1]^2)\le M_1<+\infty$; hence,  by Theorem 15.42 in \cite{Friz book}, fixed an arbitrary  $p\in \left(2(\rho+\epsilon), 4\right)$, for any  $\eta\in (0, \frac{1}{2({\rho}+\epsilon)}-\frac{1}{p})$,   there exists a constant $C_1({\rho}+\epsilon,p,M_1,\eta)$ such that 
	$$\mathbb{E}\left[\left\|S_3(W^{H_n},D)_{s,k}- \mathbb{W}_{s,k}^{H_n}\right\|_{L^q}^q\right]\leq C_1 C \, \sqrt{q} \,\delta_D^{\eta/3}. $$

	Let $F$ be a $C^1$ bounded function which is globally Lipschitz with Lipschitz constant $L$, then we have 
\begin{equation}
\label{eq:control_expected_value}
\left|\mathbb{E}\left[F(\mathbb{W}_{s,k}^{H_n})\right]-\mathbb{E}\left[F(S_3(W^{H_n},D)_{s,k})\right]\right|\leq K C_1\, C \,\delta_D^{\eta/3}.
\end{equation}
		From \eqref{eq:control_expected_value} we get
	\begin{equation}
	\begin{split}
\limsup_{n \rightarrow +\infty}\left|\mathbb{E}\left[F(\mathbb{W}_{s,k}^{H_n})\right]-
\mathbb{E}\left[F(\mathbb{W}_{s,k}^{H_{\infty}})\right] \right|\leq & \limsup_{n \rightarrow +\infty}
\left|\mathbb{E}\left[F(\mathbb{W}_{s,k}^{H_n})\right]-
\mathbb{E}\left[F(S_3(W^{H_n},D)_{s,k})\right] \right|\\
&+\limsup_{n \rightarrow +\infty}
\left|\mathbb{E}\left[F(\mathbb{W}_{s,k}^{H_{\infty}})\right]-
\mathbb{E}\left[F(S_3(W^{H_{\infty}},D)_{s,k})\right] \right|\\
&+\limsup_{n \rightarrow +\infty}
\left| \mathbb{E}\left[F(S_3(W^{H_n},D)_{s,k})\right]\right.\\
&\hspace{1.5cm}\, -\left.
\mathbb{E}\left[F(S_3(W^{H_{\infty}},D)_{s,k})\right] \right|\\
\leq & 2  L C_1 C \, \delta_D^{\eta/3}.
	\end{split}
	\end{equation}
The last step derived by the facts that $W^{H_n}$ converges in probability
	to $W^{H_{\infty}}$  and   that $S_3(W^{H_{n}},D)_{s,k}$ and $S_3(W^{H_{\infty}},D)_{s,k}$ are polynomial approximations of $W^{H_n}$ and  $W^{H_{\infty}}$, respectively. Since $\delta_D$ can be chosen in an arbitrary way the thesis is proven.

\end{proof}

\begin{remark}
Note that the restriction to  $t\in [0,1]$ is only a technical simplification, since one can always reformulate an equation on $[0,T]$ as an equation on $[0,1]$ via a reparametrization (see \cite{Friz book}). 
\end{remark}

\begin{remark}
	When $H>\frac 12$, the result can be proven following the same steps, but without the need of rough paths theory. Indeed, the solution map $W^H\to X^H$ is continuous, since we are in the framework of Young integration theory. This means that whenever $H> \frac 12$ it is sufficient to show that, for some $\alpha>\frac 12$,  when $H\to H_\infty$ it holds that $W^H\to W^{H_\infty}$ in $\C^{\alpha}([0,T])$. Latter fact can be shown again via Kolmogorov-Lamperti criterion (Corollary A.11, \cite{Friz book}). 
\end{remark}

\end{document}